\documentclass[12pt,a4paper]{amsart}
\usepackage{amsfonts}
\usepackage{amsthm}
\usepackage{amsmath}
\usepackage{amscd}
\usepackage[latin2]{inputenc}
\usepackage{t1enc}
\usepackage[mathscr]{eucal}
\usepackage{indentfirst}
\usepackage{graphicx}
\usepackage{graphics}
\usepackage{pict2e}
\usepackage{epic}
\numberwithin{equation}{section}
\usepackage[margin=2.9cm]{geometry}
\usepackage{epstopdf} 
\usepackage{tipa}
\usepackage{accents}
\usepackage{stmaryrd}

\theoremstyle{plain}
\newtheorem{Th}{Theorem}[section]

\newtheorem{Prop}[Th]{Proposition}

 \theoremstyle{definition}
\newtheorem{Def}[Th]{Definition}

\newtheorem{Rmk}[Th]{Remark}
\newtheorem{?}[Th]{Problem}

\begin{document}

\title[The Euler Class from a General Connection, Relative to a Metric]{The Euler Class from a General Connection, Relative to a Metric}

\author{Brian Klatt}

\address{Rutgers University \\ Department of Mathematics \\
New Brunswick, NJ \\ United States} 

\email{brn.kltt@gmail.com}



\begin{abstract} We extend the well-known formula for the Euler class of a real oriented even-dimensional vector bundle in terms of the curvature of a metric connection to the case of a general linear connection provided a metric is present. We rewrite the classical Gauss-Bonnet theorem in dimension two in light of this formula. We also discuss a potential application to a conjecture of Chern, and make a brief digression to discuss $m$-quasi-Einstein manifolds.
\end{abstract}

\maketitle

\section{Introduction}

The Gauss-Bonnet theorem is one of the most remarkable and beautiful theorems in geometry, and in all of mathematics. It states (see e.g. \cite{Lee}) that if $(M^2, g)$ is a closed oriented surface with metric $g$ and Gauss curvature $K_g$ then
$$\int_{M^2}\,K_g\,\mathrm{d}V=2\pi\chi(M^2)$$
where $\chi$ denotes the Euler characteristic. Thus, knowing a single geometric quantity, the total Gauss curvature, completely determines the topology of the surface; or alternatively, knowing the topology restricts the range of possibilities for a metric on the surface.

This theorem was generalized by Allendoerfer \& Weil to the case of a closed oriented $2m$-dimensional Riemannian manifold \cite{AllenWeil}, and given a profoundly simple intrinsic proof shortly after by Chern \cite{Chern}. Their result is that
$$\int_{M^{2m}}\, \mathrm{Pf}(\Omega) = (2\pi)^m\chi(M^{2m})$$
where $\Omega=\Omega^i{}_j$ is the endomorphism-valued curvature $2$-form and $$\mathrm{Pf}(\Omega)=\frac{1}{2^m m!}\displaystyle\sum_{\sigma\in S_{2m}} (-1)^{\sigma}\Omega^{\alpha_{\sigma(1)}}{}_{\alpha_{\sigma(2)}}\wedge ... \wedge \Omega^{\alpha_{\sigma(2m-1)}}{}_{\alpha_{\sigma(2m)}}$$
is a well-defined $2m$-form when computed in an oriented orthonormal basis, called the Pfaffian form.

In fact, this form makes perfect sense if $\Omega$ is the curvature $2$-form of a connection compatible with a fibre metric $g$ on an oriented rank-$2k$ vector bundle $E$, where the sum defining the Pfaffian is now taken over $S_{2k}$. We then have the well-known theorem (which we will refer to simply as the Chern-Gauss-Bonnet theorem) that
$$(2\pi)^k e(E)=\mathrm{Pf}(\Omega)=\frac{1}{2^k k!}\displaystyle\sum_{\sigma\in S_{2k}} (-1)^{\sigma}\Omega^{\alpha_{\sigma(1)}}{}_{\alpha_{\sigma(2)}}\wedge ... \wedge \Omega^{\alpha_{\sigma(2k-1)}}{}_{\alpha_{\sigma(2k)}}$$
in the cohomology ring of $M$, where $e(E)$ is the Euler class, defined as the pullback of the Thom class of $E$ by the zero section (see \cite{Bell} for a statement and proof).

It is natural to wonder whether the Euler class can be computed from a general (i.e. not necessarily metric) connection, but a famed inequality of Milnor (\cite{Milnor}, but see also \cite{Goldman}) implies that there are flat oriented plane bundles with nonzero Euler class over a surface. Thus the Euler class of such a bundle cannot be computed purely from its curvature. This does not, however, rule out the possibility of a natural formula for the Euler class from a general connection; there may exist such a formula, provided it includes non-curvature terms. It is the primary purpose of this paper to state and prove just such a formula. This we accomplish in sections 2 and 4; section 3 contains some preliminaries to the proof on the Pfaffian of matrices.

We then, in section 5, record the global and local Gauss-Bonnet formulas in light of our general formula.

Section 6 contains miscellanea related to the main theorem, in two subsections. The first part calls attention to a potential application to a conjecture of Chern on the Euler characteristic of affine manifolds.  The second is an invitation to further research on ``generalized $m$-quasi-Einstein manifolds.'' This is related to the main theorem by happenstance; it is through this topic that the author was led to discover the main theorem (in particular, by thinking about open problem (vi) of \cite{Cao}).

Finally section 7 is an appendix in two parts: The first serves to fix notations and contains some standard formulas of differential geometry so as to avoid cluttering the flow of argumentation in the main part of the paper, and the second is an extensive compendium of geometric quantities associated to a particular class of tangent bundle connections. 

\section{Statement of the Formula}

Our setup will be a general linear connection $\nabla$ and fibre metric $g$ on an oriented rank-$2k$ real vector bundle $E$. The simplest way to realize a formula of the kind we desire is to find the correct ``error term" to insert into $(2\pi)^k e(E)=\mathrm{Pf}(\Omega)$, such that the error term vanishes when the connection $\nabla$ is metric. Thus the obvious building block for our error term is $\nabla g$. This is a symmetric matrix of $1$-forms with lower bundle indices: $\nabla g=\nabla_i g_{\alpha\beta}$. We can naturally construct from this an object of the same tensorial type as $\Omega=\Omega^{\alpha}{}_{\beta}=\Omega_{ij}{}^{\alpha}{}_{\beta}$ by using the inverse metric to first construct $g^{-1}\nabla g=(g^{-1}\nabla g)_i{}^{\alpha}{}_{\beta}=g^{\alpha \gamma}\nabla_i g_{\gamma \beta}$ and then squaring this matrix, using the wedge product to multiply the one-form entries: 

\begin{Def}
The endomorphism-valued two-form $(g^{-1}\nabla g)^2$ is defined by 
\begin{align*}
(g^{-1}\nabla g)^2&=(g^{-1}\nabla g)^2)_{ij}{}^{\alpha}{}_{\beta}\\
&=(g^{-1}\nabla g)_i{}^{\alpha}{}_{\gamma}\wedge (g^{-1}\nabla g)_j{}^{\gamma}{}_{\beta}\\
&=g^{\alpha \gamma}\nabla_i g_{\gamma \delta}\,g^{\delta \epsilon}\nabla_j g_{\epsilon \beta}-g^{\alpha \gamma}\nabla_j g_{\gamma \delta}\,g^{\delta \epsilon}\nabla_i g_{\epsilon \beta}
\end{align*}
\end{Def}

With these considerations in mind, one can do explicit computations with a simple non-metric tangent bundle connection like, say, $\Gamma_i{}^k{}_j=\overline{\Gamma}_i{}^k{}_j+a\,\varphi_i\delta^k_j+b\,\varphi_j\delta^k_i+c\,\varphi^kg_{ij}$ ($\overline{\Gamma}_i{}^k{}_j$ are the Christoffel symbols for the Levi-Civita connection $\nabla_g$ of $g$ and $\varphi$ is a $1$-form) to guess the error term, which leads to our main theorem.

\begin{Th}\label{RCGB}
Let $E$ be an oriented rank-$2k$ real vector bundle with connection $\nabla$, curvature $\Omega=\Omega_{\nabla}$, and fibre metric $g$. Then

$$(2\pi)^k e(E)=\mathrm{Pf}(\Omega-\tfrac{1}{4}(g^{-1}\nabla g)^2)$$
\vspace{.05in}

\noindent where $e(E)$ is the Euler class of the bundle, and the Pfaffian is computed using any oriented orthonormal basis for the metric $g$.
\end{Th}

\begin{Rmk}
Theorem~\ref{RCGB} brings attention to an interesting class of connections $\nabla$ with respect to a given metric $g$: those satisfying $(g^{-1}\nabla g)^2=0$. This condition is strictly weaker than metric compatibility because it is satisfied if (but not only if), for example, $\nabla_i g_{\alpha \beta}=\varphi_i g_{\alpha \beta}$; on the tangent bundle, this includes Weyl connections $\nabla=\nabla_g-\tfrac{\varphi_i}{2} \delta^k_j -\tfrac{\varphi_j}{2}\delta^k_i+\tfrac{\varphi^k}{2} g_{ij}$ \cite{PedSwann}. Therefore for Weyl connections, and any connection such that $(g^{-1}\nabla g)^2=0$, one can compute the Euler class directly from the curvature without any error term via $(2\pi)^k e(E)=\mathrm{Pf}(\Omega)$.
\end{Rmk}

\section{Preliminary for the Proof}
We only need a brief interlude on properties of the Pfaffian before we can proceed to the proof. (For further reference see Volume V of \cite{Spivak}.)

\begin{Def}
Let $M_1$, ..., $M_k$ be $2k$-by-$2k$ matrices with entries in a commutative algebra $R$ over $\mathbb{Q}$. We define the Pfaffian of these matrices by
$$\mathrm{Pf}(M_1, ..., M_k)=\frac{1}{2^k k!}\displaystyle\sum_{\sigma\in S_{2k}} (-1)^{\sigma}(M_1)^{i_{\sigma(1)}}_{i_{\sigma(2)}} ... (M_k)^{i_{\sigma(2k-1)}}_{i_{\sigma(2k)}}$$
By setting all $M_i=M$ we recover $\mathrm{Pf}(M)$, the usual Pfaffian of $M$.
\end{Def}

\noindent The Pfaffian is obviously multilinear with respect to matrix addition and multiplication of matrices by elements of $R$. We also have the following

\begin{Prop}\label{Pf} The Pfaffian satisfies the following properties:

\begin{enumerate} 
\item $\mathrm{Pf}(M_1, ..., (M_i)^T, ..., M_k)=-\mathrm{Pf}(M_1,..., M_i, ..., M_k)$
\item $\mathrm{Pf}(M_1, ..., M_i, ..., M_k)=\mathrm{Pf}(M_1,..., (M_i)_A,..., M_k)$ where $(M_i)_A=\frac{1}{2}(M_i-(M_i)^T)$
\item $\mathrm{Pf}(M_1, ..., M_k)=\mathrm{Pf}((M_1)_A, ..., (M_k)_A)$ and $\mathrm{Pf}(M)=\mathrm{Pf}(M_A)$
\end{enumerate}
\end{Prop}

\begin{proof}

\begin{enumerate}
\item This is obvious from the definition, since a transposition has sign equal to $-1$.
\item With $(M_i)_A=\frac{1}{2}(M_i-(M_i)^T)$ and $(M_i)_S=\frac{1}{2}(M_i+(M_i)^T)$, compute 
\begin{align*}
\mathrm{Pf}(M_1, ..., M_i, ..., M_k)&=\mathrm{Pf}(M_1, ..., (M_i)_A+(M_i)_S, ..., M_k)\\
&=\mathrm{Pf}(M_1, ..., (M_i)_A, ..., M_k)+\mathrm{Pf}(M_1, ..., (M_i)_S, ..., M_k)
\end{align*}
\noindent and note by (1) that 
\begin{align*}
\mathrm{Pf}(M_1, ..., (M_i)_S, ..., M_k)&=\mathrm{Pf}(M_1, ..., ((M_i)_S)^T, ..., M_k)\\&
=-\mathrm{Pf}(M_1, ..., (M_i)_S, ..., M_k)
\end{align*}
\noindent so $$\mathrm{Pf}(M_1, ..., (M_i)_S, ..., M_k)=0,$$ and (2) follows.
\item Use induction on (2).
\end{enumerate}
\end{proof}

\section{Proof of the Main Theorem and Developments}

We are now ready for the proof of Thm.~\ref{RCGB}.

\subsection{Proof of Theorem \ref{RCGB}}

\begin{proof}
Let $\{E_{\alpha}\}$ be a local oriented orthonormal basis of $E$, and $\omega=\omega^{\alpha}{}_{\beta}=\omega_i{}^{\alpha}{}_{\beta}$ and $\Omega=\Omega^{\alpha}{}_{\beta}=\Omega_{ij}{}^{\alpha}{}_{\beta}$ be the connection and curvature matrices, respectively, in this frame. We split the matrices of forms
$$\omega=\omega_A+\omega_S$$
and 
$$\Omega=\Omega_A+\Omega_S$$
into antisymmetric (subscript ``A'') and symmetric (subscript ``S'') parts.

\noindent One can readily check that
$$\Omega_A=\mathrm{d}\omega_A+\omega_A\wedge\omega_A+\omega_S\wedge\omega_S$$
and
$$\Omega_S=\mathrm{d}\omega_S+\omega_A\wedge\omega_S+\omega_S\wedge\omega_A.$$

\noindent We also compute in our local orthonormal frame that

$$g^{-1}\nabla g=\delta^{\alpha \gamma}(\mathrm{d}\delta_{\gamma \beta}-\omega^{\epsilon}{}_{\gamma}\delta_{\epsilon \beta}-\omega^{\epsilon}{}_{\beta}\delta_{\gamma \epsilon})=-(\omega^{\alpha}{}_{\beta}+\omega^{\beta}{}_{\alpha})=-2\,\omega_S$$
\vspace{.01in}

\noindent which immediately implies
$$\tfrac{1}{4}(g^{-1}\nabla g)^2=\omega_S\wedge\omega_S.$$
Therefore
$$\mathrm{Pf}(\Omega-\tfrac{1}{4}(g^{-1}\nabla g)^2)=\mathrm{Pf}(\mathrm{d}\omega_A+\omega_A\wedge\omega_A+\Omega_S)=\mathrm{Pf}(\mathrm{d}\omega_A+\omega_A\wedge\omega_A)$$
where we neglected the symmetric part in the last equality by (3) of Prop.~\ref{Pf}.

All that remains of the proof is to show that $\omega_A$ defines a metric connection, for then
$$\mathrm{Pf}(\Omega-\tfrac{1}{4}(g^{-1}\nabla g)^2)=\mathrm{Pf}(\mathrm{d}\omega_A+\omega_A\wedge\omega_A)=\mathrm{Pf}(\Omega_{\omega_A})=(2\pi)^k e(E)$$ by the Chern-Gauss-Bonnet theorem. This follows easily from
$$\omega_A=\omega-\omega_S=\omega+\tfrac{1}{2}g^{-1}\nabla g.$$
In fact, define $\nabla^g=\nabla+\frac{1}{2}g^{-1}\nabla g$. This is clearly a connection with connection matrix $\omega_A$ in our local orthonormal frame, and since this matrix is antisymmetric, it's a metric connection. This completes the proof.
\end{proof}

\subsection{The Canonical Associated Metric Connection}

We record an observation from the proof.

\begin{Def} 
Suppose $E$ is a real vector bundle with linear connection $\nabla$ and fibre metric $g$. Then we call
$$\nabla^g=\nabla+\tfrac{1}{2}g^{-1}\nabla g$$

\vspace{.05in}
\noindent the \emph{canonical g-metric connection associated to} $\nabla$.
\end{Def}

\begin{Prop}
$\nabla^g$ is indeed a metric connection with respect to $g$.
\end{Prop}

\noindent We give a proof independent of our work in the proof of Thm~\ref{RCGB}.

\begin{proof}
Choose an arbitrary local frame $\{E_{\alpha}\}$ and let $\omega^{\alpha}{}_{\beta}$ be the matrix of connection $1$-forms of $\nabla$. Then the connection $1$-forms of $\nabla^g$ are $\eta^{\alpha}{}_{\beta}=\omega^{\alpha}{}_{\beta}+\frac{1}{2}g^{\alpha \gamma}\nabla g_{\gamma \beta}$, and we can simply compute
\begin{align*}
\nabla^g g&=\mathrm{d}g_{\alpha \beta}-\eta^{\gamma}{}_{\alpha}g_{\gamma \beta}-\eta^{\gamma}{}_{\beta}g_{\alpha \gamma}\\
&=\mathrm{d}g_{\alpha \beta}-(\omega^{\gamma}{}_{\alpha}+\tfrac{1}{2}g^{\gamma \delta}\nabla g_{\delta \alpha})g_{\gamma \beta}-(\omega^{\gamma}{}_{\beta}+\tfrac{1}{2}g^{\gamma \delta}\nabla g_{\delta \beta})g_{\alpha \gamma}\\
&=\mathrm{d}g_{\alpha \beta}-\omega^{\gamma}{}_{\alpha}g_{\gamma \beta}-\omega^{\gamma}{}_{\beta}g_{\alpha \gamma}-\nabla g_{\alpha \beta}\\
&=0
\end{align*}
\end{proof}

In brief, then, Thm.~\ref{RCGB} tells us that to compute the Euler class from $\nabla$ in the presence of a metric $g$, one should compute the Pfaffian form of the canonical metric connection induced by $\nabla$ and $g$.

Our next proposition gives an alternative characterization of the connection $\nabla^g$ as the connection nearest to $\nabla$, at each point, in the affine subspace of $g$-metric connections. Thus $\nabla^g$ is a sort of orthogonal projection of $\nabla$ onto the space of $g$-metric connections.

\begin{Prop}\label{proj}
Let $E$ be a real vector bundle, $\nabla$ a linear connection, and $g$ a fibre metric. Also, arbitrarily fix a Riemannian metric $h$ on the underlying manifold, and define length in $T^*M\otimes E\otimes E^*$ with the metric $\mu=h^{-1}\otimes g\otimes g^{-1}$. If $\nabla'$ is a $g$-metric connection, then $|\nabla-\nabla^g|\leq |\nabla-\nabla'|$ with equality if and only if $\nabla'=\nabla^g$.
\end{Prop}

\begin{proof}
The basic point is that $\nabla^g-\nabla=\frac{1}{2}g^{-1}\nabla g$ is $g$-symmetric, which is orthogonal to the $g$-antisymmetric directions that define the metric connections.

To be precise, let $\nabla'$ be a $g$-metric connection. Then
$$|\nabla-\nabla'|^2=|(\nabla-\nabla^g)-(\nabla'-\nabla^g)|^2$$
We'll compute in a local $g$-orthonormal frame $\{E_{\alpha}\}$ of $E$ and a local $h$-orthonormal frame $\{e_i\}$ of $TM$, so $g_{\alpha \beta}=\delta_{\alpha \beta}$ and $h_{ij}=\delta_{ij}$. Denote by $\omega$ the matrix of connection $1$-forms. As discussed in the proof of Thm~\ref{RCGB}, $\nabla-\nabla^g=\omega_S$, a symmetric matrix of $1$-forms, while $\nabla'-\nabla^g=T_i{}^{\alpha}{}_{\beta}$ is an antisymmetric matrix of $1$-forms since each connection in the difference is metric. Now compute
$$\mu(\omega_S, T)=\displaystyle\sum_{i, \,\alpha, \,\beta}(\omega_S)_i{}^{\alpha}{}_{\beta} \,T_i{}^{\alpha}{}_{\beta}=-\displaystyle\sum_{i, \,\alpha, \,\beta}(\omega_S)_i{}^{\beta}{}_{\alpha} \,T_i{}^{\beta}{}_{\alpha}=-\mu(\omega_S, T)$$
Thus $\mu(\omega_S, T)=0$ and we find that
\begin{align*}
|\nabla-\nabla'|^2&=|(\nabla-\nabla^g)-(\nabla'-\nabla^g)|^2\\
&=|(\nabla-\nabla^g)|^2+|(\nabla'-\nabla^g)|^2\\
&=|\omega_S|^2+|(\nabla'-\nabla^g)|^2
\end{align*}
and so the conclusion clearly follows.
\end{proof}

\section{The Formula on the Tangent Bundle of a Surface}
We will now give Thm~\ref{RCGB} explicitly on the tangent bundle of a compact oriented surface, so as to compare it to the classical Gauss-Bonnet formula. Of course the integrands will differ by a divergence term; we will determine it exactly. We also give the corresponding local Gauss-Bonnet theorem.

\subsection{Global Gauss-Bonnet Theorem}
We begin by specializing our setup to the case where $E=TM$, and so $g$ is a Riemannian metric on $M$, with Levi-Civita connection $\nabla_g$. Given a connection $\nabla$ (not assumed to be metric or torsion-free) we have a ``right triangle'' (by Prop.~\ref{proj}) of connections with vertices $\nabla$, $\nabla^g$, and $\nabla_g$. If we write $\nabla=\nabla_g-D$ (so $D$ is the hypotenuse of the triangle), then straightforward computation from the definition of $\nabla^g$ shows $\nabla^g=\nabla+D_S$ and therefore $\nabla^g=\nabla_g-D_A$ where
$$D_S=\tfrac{1}{2}(D+g^{-1}D^Tg)=\tfrac{1}{2}(D_i{}^k{}_j+g^{kl}D_i{}^m{}_l g_{mj})$$
and
$$D_A=\tfrac{1}{2}(D-g^{-1}D^Tg)=\tfrac{1}{2}(D_i{}^k{}_j-g^{kl}D_i{}^m{}_l g_{mj})$$

Now we set $B=D_A$ and $H=\Omega(\nabla^g)$ for convenience, so we have in particular $\nabla^g=\nabla_g-B$. The difference tensor $B$ has one independent trace $B_i=B_k{}^k{}_i=B_{kki}=\tfrac{1}{2}(D_{jji}-D_{jij})$, and standard differential geometric calculations (see the appendix) reveal that
\begin{align*}
\mathrm{Pf}(H)&=(K_g-\mathrm{d}^*_g(B_i))\mathrm{d}V_g\\
&=(K_g-\tfrac{1}{2}\mathrm{d}^*_g(D_{jji}-D_{jij}))\mathrm{d}V_g
\end{align*}

\noindent which by Thm.~\ref{RCGB} yields the following

\begin{Prop}[Global Gauss-Bonnet Theorem]\label{globGB}
Suppose $\nabla$ is a connection with curvature $\Omega$ on the tangent bundle of a closed oriented surface $(M^2, g)$ with Levi-Civita connection $\nabla_g$, Gauss curvature $K_g$, and volume form $\mathrm{d}V_g$. If we write $\nabla=\nabla_g-D$ and $\nabla^g=\nabla_g-B$, then
\begin{align*}
2\pi\, e(TM^2)&=\mathrm{Pf}(\Omega-\tfrac{1}{4}(g^{-1}\nabla g)^2)\\
&=(K_g-\mathrm{d}^*_g(B_i))\mathrm{d}V_g\\
&=(K_g-\tfrac{1}{2}\mathrm{d}^*_g(D_{jji}-D_{jij}))\mathrm{d}V_g
\end{align*}
and so
$$\displaystyle\int_{M^2} (K_g-\tfrac{1}{2}\mathrm{d}^*_g(D_{jji}-D_{jij}))\mathrm{d}V_g=2\pi\,\chi(M^2)$$
\end{Prop}

\vspace{.1in}

\subsection{Local Gauss-Bonnet Theorem}
\noindent Considering the divergence theorem with inward-pointing normal $N$,
$$\int_D -\mathrm{d}^*_g(B_i)\mathrm{d}V_g=\int_D \mathrm{div}_g(B^i)\mathrm{d}V_g=-\int_{\partial D} g_{ij} B^i N^j \mathrm{d}s=-\int_{\partial D} B_i(N) \mathrm{d}s$$
the following proposition is the natural local version of the above Global Gauss-Bonnet theorem (we borrow some phrasing and notation from Theorem 9.3 of \cite{Lee}).

\begin{Prop}[Local Gauss-Bonnet Theorem]\label{locGB}
Suppose $\gamma$ is a curved polygon on an oriented surface $(M^2, g)$, positively oriented as the boundary of an open set $D$ with compact closure, with inward pointing normal $N$, and exterior angles $\epsilon_i$. Let $\overline{\nabla}$ be the Levi-Civita connection, $\kappa_N=g(\overline{\nabla}_{\dot{\gamma}}\dot{\gamma}, N)$ be the signed curvature of $\gamma$, and $B_i$ be a $1$-form. Then
\begin{equation}\label{locGB}
\int_D (K_g-\mathrm{d}^*_g(B_i)) \mathrm{d}A_g + \int_{\gamma} (\kappa_N+B_i(N)) \mathrm{d}s + \displaystyle\sum_i \epsilon_i = 2\pi
\end{equation}
If $\nabla'=\overline{\nabla}-B_i{}^k{}_j$ is a metric connection, i.e. $B_{ikj}=g_{km}B_i{}^m{}_j$ satisfies $B_{ikj}=-B_{ijk}$,  and $B_i=B_k{}^k{}_i=B_{kki}$, then $K_g-\mathrm{d}^*_g(B_i)=\mathrm{Pf}(\Omega_{\nabla'})$ is the ``Gauss curvature'' of $\nabla'$ and $\kappa_N+B_i(N)=g(\nabla'_{\dot{\gamma}}\dot{\gamma}, N)$ is the ``signed curvature of $\gamma$ with respect to $\nabla'$.'' If $\nabla=\overline{\nabla}-D_i{}^k{}_j$ is a general connection, then according to Thm.~\ref{RCGB} we apply this to the metric connection $\nabla^g=\overline{\nabla}-\tfrac{1}{2}(D_i{}^k{}_j-g^{kl}D_i{}^m{}_l g_{mj})$.
\end{Prop}

\begin{proof}
For any $1$-form $B_i$, Eqn~\ref{locGB} follows from the standard local Gauss-Bonnet theorem, the divergence theorem (i.e. Stokes' theorem), and an argument to approximate $D$ by smooth domains like the one at the end of the proof of the Gauss-Bonnet formula (Theorem 9.3) in \cite{Lee}.

If $\nabla'=\overline{\nabla}-B_i{}^k{}_j$ is a metric connection, then $K_g-\mathrm{d}^*_g(B_i)=\mathrm{Pf}(\Omega_{\nabla'})$ is just Eqn.~\ref{H1} of the appendix. We can also compute
$$g(\nabla'_{\dot{\gamma}}\dot{\gamma}, N)=g(\overline{\nabla}_{\dot{\gamma}}\dot{\gamma}-B_{\dot{\gamma}}\dot{\gamma}, N)=\kappa_N-g(B_{\dot{\gamma}}\dot{\gamma}, N)$$
where in any orthonormal basis, say $\{\dot{\gamma}, N\}$, $B=
[\begin{smallmatrix}
0 & b\\
-b & 0
\end{smallmatrix}]$ as an endomorphism-valued $1$-form, with $b=-\star_1 B_i$ (see Eqn.~\ref{b&B} in the Appendix). Thus $$B_{\dot{\gamma}}\dot{\gamma}=-b(\dot{\gamma})N=-(\star_1 b)(N) N=-B_i(N)N,$$ which yields
$$g(\nabla'_{\dot{\gamma}}\dot{\gamma}, N)=\kappa_N+B_i(N)$$
\end{proof}

\begin{Rmk}
We will not go into the details, but the interested reader may wish to check the following: if one tracks the proof of Gauss-Bonnet given in \cite{Lee} but substitutes a general connection for the Levi-Civita connection and writes this connection as a metric connection plus difference tensor, any contribution due to the difference tensor cancels out of the calculations. One is thereby left with Prop.~\ref{locGB}. It is precisely when the metric connection is the canonical associated metric connection that we are being ``least wasteful'' in what gets cancelled out. 
\end{Rmk}

\begin{Rmk}
We presume that Prop.~\ref{locGB} is not really ``new.'' For instance, Corwin and Morgan in \cite{CorwinMorgan} prove a Gauss-Bonnet theorem on a smooth disk in a ``surface with densities'': $\mathrm{d}s=\delta_1\mathrm{d}s_0$ and $\mathrm{d}A=\delta_2\mathrm{d}A_0$ where the $\delta_i$ are positive density functions. They define $K'=K_g-\Delta\log\delta_1$ and $\kappa=\tfrac{\delta_1}{\delta_2}\kappa_N-\tfrac{1}{\delta_2}g(\nabla\delta_1, N)$ and show that
$$\int_{\gamma} \tfrac{\delta_2}{\delta_1}\kappa \,\mathrm{d}s_0+\int_D K' \,\mathrm{d}A_0=2\pi$$
However, this is clearly equivalent to
$$\int_{\gamma} (\kappa_N-\mathrm{d}(\log\delta_1)(N)) \,\mathrm{d}s_0+\int_D (K_g+\mathrm{d}^*_g(\mathrm{d}\log\delta_1)) \,\mathrm{d}A_0=2\pi$$
since $\Delta=-\mathrm{d}^*_g\mathrm{d}$. This is Prop.~\ref{locGB} when $B_i=-\mathrm{d}(\log\delta_1)$. If we write $-\log\delta_1=f$, then one nice choice of connection $\nabla$ whose canonical metric connection yields this $B_i$ is
$$\nabla=\overline{\nabla}+(m-1)\mathrm{d}f_i\,\delta^k_j+(m-1)\mathrm{d}f_j\,\delta^k_i+(m+1)\overline{\nabla}f^k g_{ij}$$
It's ``nice'' because it's torsion-free and the vanishing of its traceless Ricci tensor (with respect to $g$) is related to the surface being a generalized $m$-quasi-Einstein manifold, a notion which seems to be naturally related to that of ``manifold with density.'' See e.g. \cite{Morgan} and \cite{Case}, and the next section.
\end{Rmk}

\section{Applications \& Explorations}

Given that $\mathrm{Pf}(\Omega-\frac{1}{4}(g^{-1}\nabla g)^2)=\mathrm{Pf}(\Omega_{\nabla^g})$ and $\nabla^g$ is a metric connection, one may legitimately ask why we have bothered to frame Thm.~\ref{RCGB} in terms of general connections. The real reason is that this is the form in which the theorem was discovered, and we prefer it aesthetically. However we also believe there are substantial reasons to frame the theorem this way, as follows.

For one, we may be given a connection with special properties a priori unrelated to a metric and want to explore the topology of our manifold; we can then introduce a metric and use Thm.~\ref{RCGB} to do so. One such instance is discussed in the next subsection. 

Another possibility is that we could have a metric and some other differential geometric data which is not obviously subordinate to the metric. It is not clear in such a scenario that one should default to the Levi-Civita connection, or any particular metric connection for that matter; perhaps some other, non-metric, connection is better adapted to the situation, or can suggest natural geometric quantities to study.  If so, Thm.~\ref{RCGB} can be used to explore the relation between the geometric data and the Euler characteristic. A potential example of this is discussed in the second subsection below.

\subsection{Chern's Conjecture on the Euler Characteristic of Affine Manifolds}

An advantage of phrasing Thm.~\ref{RCGB} in terms of general connections is that we hope it may find an application in the study of a conjecture of Chern. This conjecture posits that the Euler characteristic of a closed manifold with a torsion-free flat connection on its tangent bundle should have vanishing Euler characteristic (see \cite{Martinez} for an overview). If the connection is metric then the conjecture is true, even without the torsion-free assumption, by the Chern-Gauss-Bonnet theorem. The conjecture is also true in dimension two by Milnor's inequality \cite{Milnor}, again without the torsion-free assumption. However, examples of Smillie in all even dimensions greater than two (\cite{Smillie}, \cite{Goldman}) show that in general the torsion-free assumption is necessary for the truth of the conjecture. 

By Thm.~\ref{RCGB}, it suffices to show that $\mathrm{Pf}((g^{-1}\nabla g)^2)$ integrates to $0$ for any Riemannian metric $g$ when $\nabla$ is a torsion-free flat connection; or equivalently to show that $\mathrm{Pf}((g^{-1}\nabla g)^2)$ is exact. The author has made attempts at this in dimension two without noteworthy success. Perhaps instead of working with an arbitrary metric $g$, one should try to be more discerning and choose a $g$ adapted in some way to $\nabla$. Imposing $\nabla g=0$ is too strong, but what other equations are natural? One obvious idea based on the concepts discussed here is to demand that the metric $g$ minimize the distance between $\nabla$ and $\nabla^g$ globally, or what amounts to the same, minimize the magnitude of $\nabla g$. This suggests minimizing a scale-invariant functional like
$$\mathcal{F}_{\nabla}(g)=\int_{M^n} |\nabla g|^n\,\mathrm{d}V_g$$
or minimizing, say, the functional
$$\mathcal{K}_{\nabla}(g)=\int_{M^n} |\nabla g|^2\,\mathrm{d}V_g$$
subject to the constraint $\int_{M^n} \mathrm{d}V_g=1$. (Note that $\mathcal{F}_{\nabla}=\mathcal{K}_{\nabla}$ when $n=2$, so we don't need the volume constraint when using $\mathcal{K}_{\nabla}$ on a surface; is this possibly related to not needing the torsion-free assumption when $n=2$?) Perhaps the minima of these functionals, if they exist, will be well-adapted to the problem at hand. We will not, however, discuss these ideas any further here, and merely hope that they may prove fruitful in future efforts on this conjecture.

\subsection{Natural Connections Incompatible with a Given Metric: Generalized $m$-Quasi-Einstein Manifolds}

In this subsection, we will discuss a situation in which a metric $g$ occurs alongside other differential geometric data, making the introduction of a connection $\nabla$ incompatible with $g$ plausibly natural.

\subsubsection{Generalized $m$-quasi-Einstein Manifolds}

A \emph{generalized $m$-quasi-Einstein manifold} \cite{Barros} is a Riemannian manifold such that
\begin{equation}\label{GmQE}
r_g+\tfrac{1}{2}\pounds_{\varphi^{\#}} g-\tfrac{1}{m}\varphi\otimes \varphi=\lambda g
\end{equation}
where $r_g$ is the Ricci tensor of $g$, $\varphi\in\Omega^1(M^n)$, $\pounds$ denotes the Lie derivative, $m$ is an extended real number, and $\lambda\in C^{\infty}(M^n)$. To understand the significance of this equation, it helps to look at some special cases. When $m=0$ we can understand, by convention, that the equation simplifies to $r_g=\lambda g$, which also occurs when $\varphi=0$ for any $m$; this is an Einstein manifold when $n\ge 3$. When $m=\pm\infty$ we have $\frac{1}{m}=0$, so if $\lambda$ is constant, Eqn.~\ref{GmQE} becomes the equation of a Ricci soliton, a notion which occurs in the study of singularities of the Ricci flow \cite{Cao}. In the case $m=-(n-2)$, the equation describes an Einstein-Weyl manifold (compare to the equation on page 100 of \cite{PedSwann}), a notable concept in conformal geometry. Finally, when $\varphi=\mathrm{d}f$, $m$ is a non-zero integer, and $\lambda$ is constant, Eqn.~\ref{GmQE} is related to constructing Einstein warped products and conformal Einstein metrics, depending on $m$ (see \cite{Case} for more information and references).

As a Riemannian geometer working with Eqn.~\ref{GmQE}, it is tempting to default to using the Levi-Civita connection $\nabla_g$, which perhaps gives undue primacy to the metric $g$ and places $\varphi$ in an undeservedly subordinate role. We propose that one way to unify these two objects which places them on more equal footing is to consider a connection $\nabla$ which depends on both quantities and generates Eqn.~\ref{GmQE} in an appropriate sense. Namely, we demand that Eqn.~\ref{GmQE} be equivalent to $\mathring{r}_{(ij)}=0$, where $\mathring{r}_{()}$ is the traceless (with respect to $g$) symmetric part of the Ricci tensor of $R=\Omega(\nabla)$. (Our perspective here is strongly influenced by the philosophy, inherited from E. Cartan and espoused in \cite{Sharpe}, that differential geometry is the study of connections on principal bundles.)

\subsubsection{The Connection Ansatz}

What would this connection $\nabla$ look like? Answering this question comes down to determining the difference tensor $\nabla-\nabla_g$. It's easy to see based on Eqn.~\ref{GmQE} that this three-index difference tensor should be a tensor algebraic in $g$ and $\varphi$. Denoting the connection coefficients of $\nabla$ by $\Gamma_i{}^k{}_j$, and the Christoffel symbols of the Levi-Civita connection $\nabla_g$ by $\overline{\Gamma}_i{}^k{}_j$, our ansatz is
\begin{equation}\label{ansatz}
\Gamma_i{}^k{}_j=\overline{\Gamma}_i{}^k{}_j+a\varphi_i \delta^k_j+b\varphi_j \delta^k_i+c\varphi^k g_{ij}
\end{equation}

\noindent The torsion of $\nabla$ vanishes if and only if $a=b$ but, without a full geometric understanding of the issue, we prefer to leave open the possibility that torsion may play some role.

The formula for the traceless Ricci tensor of this connection is (see the Appendix)
\begin{align*}
\mathring{r}_{()}=&\,\overline{r}-\tfrac{1}{2}((n-1)b+c)\pounds_{\varphi^{\#}} \,g+((n-1)b^2-c^2)\varphi\otimes \varphi+\\
&-\tfrac{1}{n}(\overline{s}+((n-1)b+c)\mathrm{d}^*_g\varphi+((n-1)b^2-c^2)|\varphi|^2)g
\end{align*}

\noindent From this we see clearly that Eqn~\ref{GmQE} is identical to $\mathring{r}_{()}=0$ provided
$$(n-1)b+c=-1$$
$$(n-1)b^2-c^2=-\tfrac{1}{m}$$
or equivalently
$$c=-1-(n-1)b$$
$$(n-1)(n-2)b^2+2(n-1)b+(1-\tfrac{1}{m})=0$$
When $n=2$, these equations have the unique solution
$$b=-\tfrac{1}{2}(1-\tfrac{1}{m})$$
$$c=-\tfrac{1}{2}(1+\tfrac{1}{m})$$
When $n>2$, the discriminant of the quadratic satisfied by $b$ is $\Delta=4(n-1)(1+\tfrac{n-2}{m})$ so only when $m=-(n-2)$ is there a unique solution, given by
$$b=-\tfrac{1}{n-2}$$
$$c=\tfrac{1}{n-2}$$
When $a=b$, so that torsion vanishes, this is precisely the case in which the manifold is an Einstein-Weyl manifold, and the connection $\nabla$ is torsion-free and preserves the conformal class of the metric $g$.

Otherwise, there is no unique solution and given Eqn.~\ref{GmQE} we can construct \emph{two} torsion-free connections with our ansatz such that Eqn~\ref{GmQE} is equivalent to $\mathring{r}_{()}=0$. The explicit solutions are
$$b=\tfrac{1}{n-2}\big(-1\pm\sqrt{\tfrac{m+n-2}{m(n-1)}}\,\big)$$
$$c=-\tfrac{1}{n-2}\big(-1\pm\sqrt{\tfrac{(m+n-2)(n-1)}{m}}\,\big)$$
Intriguingly, the radical in the expression for $b$ (which is equivalent to $b+c$, see below) seems to occur in the integration of certain differential equations when constructing a class of quasi-Einstein manifolds; see the proof of Theorem 5.7 in \cite{Case}.

Despite the non-uniqueness, there is one particularly noteworthy case which has been singled-out previously. This occurs when $a=b$ and $c=0$, which identifies the torsion-free connections among the family \ref{ansatz} which are projectively equivalent to (i.e. have the same unparametrized geodesics as) the Levi-Civita connection. From our system of equations for $b$, $c$, and $m$ we find the values $b=-\tfrac{1}{n-1}$ and $m=-(n-1)$. This was studied by Wylie and Yeroshkin in \cite{WylieYer}, primarily in the case where $\varphi$ is an exact form. Their interest seems to have been driven by the fact that when $a=b=-\tfrac{1}{n-1}$, $c=0$, and $\varphi$ is closed, the Ricci tensor of the torsion-free connection $\nabla$ is symmetric and particularly simple, satisfying (see the appendix)
$$r=r_g+\tfrac{1}{2}\pounds_{\varphi^{\#}}g+\tfrac{1}{n-1}\varphi\otimes\varphi$$
This is one of the so-called Bakry-\'{E}mery Ricci tensors.

We should also note that the family \ref{ansatz} with $a=b$ and $\varphi$ an exact form was noted in \cite{LiXia} along with a computation of the Ricci tensor, but no connection was made to $m$-quasi-Einstein manifolds. (The author considered the connections \ref{ansatz} for the reasons outlined above before becoming aware of \cite{WylieYer} or \cite{LiXia}.)

We hope that this brief introduction might interest others in investigating the connections defined by Eqn.~\ref{ansatz} and what role they might play, if any, in the study of quasi-Einstein manifolds. It's possible that they are only noteworthy in the cases already studied (that is, the torsion-free cases in which the connection preserves at least the conformal class of the metric, or is projectively equivalent to the Levi-Civita connection), but the general case is likely worthy of further scrutiny. One possible point of entry to evaluating the significance of the connection $\nabla$ when it is torsion-free, given the attention already afforded to the conformal and projective special cases, is to write,
\begin{align*}
\Gamma_i{}^k{}_j=\,\,&\overline{\Gamma}_i{}^k{}_j-c\varphi_i \delta^k_j-c\varphi_j \delta^k_i+c\varphi^k g_{ij}\\
&+(b+c)\varphi_i \delta^k_j+(b+c)\varphi_j \delta^k_i
\end{align*}
where we can easily compute
$$b+c=\mp\sqrt{\tfrac{m+n-2}{m(n-1)}}$$
This realizes the connection $\nabla$ as being projectively equivalent to a connection 
$$\nabla'=\nabla_g-c\varphi_i \delta^k_j-c\varphi_j \delta^k_i+c\varphi^k g_{ij}$$
which preserves the conformal class of $g$, and such that the $1$-form which determines the rescaling of the metric $g$ under $\nabla'$-parallel transport is a constant multiple of the $1$-form which determines the reparametrization of the $\nabla'$-geodesics. When $\varphi=\mathrm{d}f$, we have
$$\nabla=\nabla_{e^{-2cf}g}\mp\sqrt{\tfrac{m+n-2}{m(n-1)}}(\partial_i f \delta^k_j+\partial_j f \delta^k_i)$$
which seems to single out the conformal metric $e^{-2cf}g$ for consideration. (Though, recall that there are really \emph{two} values of $c$ which are roots of a quadratic, so we have conformal \emph{metrics} to consider; this suggests that additionally the midpoint $\tfrac{1}{n-2}$, and the corresponding conformal metric $e^{\frac{-2f}{n-2}}g$, may be worthy of consideration. This is perhaps further suggested by the fact that if one writes $\mathring{r}_{(ij)}=0$ in terms of $\tilde{g}=e^{\frac{-2f}{n-2}}g$ quantities, then second derivatives of $f$ don't appear explicitly.)

The author has not fully pursued these ideas in favor of applying Thm.~\ref{RCGB} with the connections in Eqn.~\ref{ansatz} to investigate the Euler characteristic of generalized $m$-quasi-Einstein $4$-manifolds; this will appear in a future work. We encourage others to take up the general line of thought in this section, and have provided in the appendix an extensive compendium of formulae as an aid to those interested.

\section{Appendix}

\subsection{Relations Between Curvatures} On a general vector bundle with connections $\nabla$ and $\nabla'$ related by $\nabla=\nabla'-B$, we can easily calculate that the relation
\begin{equation}\label{Curvatures}
\Omega_{\nabla}=\Omega_{\nabla'}-\nabla'\circ B+B\wedge B
\end{equation}
exists between the curvatures. Note that $\nabla'\circ B=\mathrm{d}B+\omega'\wedge B+B\wedge\omega'$ denotes the $\nabla'$-covariant derivative of $B$ as an $\mathrm{End}(E)$-valued $1$-form; thus it is an $\mathrm{End}(E)$-valued $2$-form. Specializing to $E=TM$ one computes that
$$(\nabla'\circ B)_{ij}{}^k{}_l=\nabla'_i B_j{}^k{}_l-\nabla'_j B_i{}^k{}_l+T(\nabla')_i{}^m{}_j \,B_m{}^k{}_l$$
where $T(\nabla')=\omega'_i{}^k{}_j-\omega'_j{}^k{}_i$ is the torsion of $\nabla'$.

Suppose now that $(M, g)$ is a Riemannian manifold, $\tilde{\nabla}$ is a metric connection on the tangent bundle, and $\overline{\nabla}=\nabla_g$ is the Levi-Civita connection. Write $\tilde{\nabla}=\overline{\nabla}-B$, and denote the curvatures by $H=\Omega(\tilde{\nabla})$, $\overline{R}=\Omega(\overline{\nabla})$. Then the formula above reads
$$H_{ij}{}^k{}_l=\overline{R}_{ij}{}^k{}_l-\overline{\nabla}_i B_j{}^k{}_l+\overline{\nabla}_j B_i{}^k{}_l+B_i{}^k{}_m B_j{}^m{}_l-B_j{}^k{}_m B_i{}^m{}_l$$
We can convert this into a formula with all indices down by defining $H_{ijkl}=g_{km}H_{ij}{}^m{}_l$, $\overline{R}_{ijkl}=g_{km}\overline{R}_{ij}{}^m{}_l$, and $B_{ikj}=g_{km}B_i{}^m{}_j$ yielding
\begin{equation}\label{curvs}
H_{ijkl}=\overline{R}_{ijkl}-\overline{\nabla}_i B_{jkl}+\overline{\nabla}_j B_{ikl}+B_{ikm} B_{jml}-B_{jkm} B_{iml}
\end{equation}
where we're using an extended Einstein summation convention such that summation over repeated lower indices is understood and denotes contraction with the inverse metric.

\subsubsection{On the Tangent Bundle of a Surface}
Consider now Eqn~\ref{curvs} on a surface, where we work in a local orthonormal basis $\{e_1, e_2\}$. We have only the single scalar equation
$$H_{1212}=\overline{R}_{1212}-\overline{\nabla}_1 B_{212}+\overline{\nabla}_2 B_{112}+B_{11m} B_{2m2}-B_{21m} B_{1m2}$$
Considering the skew-symmetry of $B$ in the second and third indices (since it's the difference between metric connections), each of the last two terms vanishes, so this is in fact just
$$H_{1212}=K_g-\overline{\nabla}_1 B_{212}+\overline{\nabla}_2 B_{112}$$
by also substituting in the Gauss curvature $K_g$ of $g$. Now we trace $-\overline{\nabla}_i B_{jkl}+\overline{\nabla}_j B_{ikl}$ over $i, k$ and $j, l$ to find
$$-\overline{\nabla}_i B_{jij}+\overline{\nabla}_j B_{iij}=-\overline{\nabla}_1 B_{212}+\overline{\nabla}_2 B_{112}-\overline{\nabla}_2 B_{121}+\overline{\nabla}_1 B_{221}=2(-\overline{\nabla}_1 B_{212}+\overline{\nabla}_2 B_{112})$$
by using the skew-symmetry of $B$, while also 
$$-\overline{\nabla}_i B_{jij}+\overline{\nabla}_j B_{iij}=2\overline{\nabla}_i B_{jji}=-2\,\mathrm{d}^*_g(B_{jji})=-2\,\mathrm{d}^*_g(B_i)$$ by again using skew-symmetry and renaming dummy indices in the first equality, defining $B_{jji}=B_i$ in the last, and where we are denoting by $\mathrm{d}^*_g$ the formal $L^2$-adjoint of $\mathrm{d}$ which is induced by $g$; \,$\mathrm{d}^*_g\psi=-\star_2\mathrm{d}\star_1\psi=-\nabla_i \psi_i$ on $1$-forms, where $\star_i$ is the Hodge star operator on $i$-forms. Thus
\begin{equation}\label{H1}
H_{1212}=K_g-\mathrm{d}^*_g(B_i)
\end{equation}

Alternatively we can proceed from Eqn.~\ref{Curvatures}
$$H=\bar{R}-\nabla'\circ B+B\wedge B=\bar{R}-(\mathrm{d}B+\bar{\Gamma}\wedge B+B\wedge\bar{\Gamma})+B\wedge B$$
in our orthonormal frame $\{e_1, e_2\}$, in which case
$$\bar{\Gamma}=
\begin{bmatrix}
0 & \bar{\gamma}\\
-\bar{\gamma} & 0
\end{bmatrix}, \hspace{.1in}
B=
\begin{bmatrix}
0 & b\\
-b & 0
\end{bmatrix}$$
where $\bar{\gamma}$ is a locally-defined $1$-form and $b=\mathrm{Pf}(B)$ is a globally-defined $1$-form if $M^2$ is oriented and our frame is consistent with the orientation, but $b$ generally has an ambiguous sign. It's trivial to compute that  $\bar{\Gamma}\wedge\bar{\Gamma}$, $B\wedge B$, and the (super)commutator $\bar{\Gamma}\wedge B+B\wedge\bar{\Gamma}$ vanish, so
\begin{align*}
H&=\bar{R}-\mathrm{d}B\\
&=\mathrm{d}\bar{\Gamma}-\mathrm{d}B\\
&=\begin{bmatrix}
0 & \mathrm{d}\bar{\gamma}-\mathrm{d}b\\
-(\mathrm{d}\bar{\gamma}-\mathrm{d}b) & 0
\end{bmatrix}\\
&=\begin{bmatrix}
0 & K_g\mathrm{d}V_g-\mathrm{d}b\\
-(K_g\mathrm{d}V_g-\mathrm{d}b) & 0
\end{bmatrix}
\end{align*}
Comparing to Eqn~\ref{H1} reveals $$\mathrm{d}b=\mathrm{d}^*_g(B_i)\mathrm{d}V_g;$$
this reflects, and is a consequence of, the easily computed fact that 
\begin{equation}\label{b&B}
b=-\star_1 B_i \hspace{.1in}\Leftrightarrow\hspace{.1in} \star_1 b=B_i
\end{equation}
which sheds more light on our earlier observation that $b$ has an ambiguous local sign.

In any case, putting this all together,
\begin{align*}
\mathrm{Pf}(H)&=H_{1212}\,e^1\wedge e^2\\
&=(K_g-\mathrm{d}^*_g(B_i))\mathrm{d}V_g
\end{align*}

\subsection{The Curvature of the Levi-Civita Connection Plus a Pure-Trace Tensor}

Let $\overline{\Gamma}_i{}^k{}_j$ denotes the Christoffel symbols of the Levi-Civita connection $\nabla_g$ of a metric $g$. We consider a connection $\nabla$ whose coefficients are given by
$$\Gamma_i{}^k{}_j=\overline{\Gamma}_i{}^k{}_j+\alpha_i \delta^k_j+\beta_j \delta^k_i+\gamma^k g_{ij}$$
where $\alpha$, $\beta$, $\gamma$ are $1$-forms. We will record various quantities associated to this connection.

\subsubsection{Torsion, Non-Metricity, Canonical Metric Connection}

The torsion is
$$T(\nabla)=(\alpha-\beta)_i \delta^k_j-(\alpha-\beta)_j \delta^k_i$$
 which clearly vanishes exactly when $\alpha=\beta$,
while the non-metricity is
$$\nabla_i g_{jk}=-2\alpha_i g_{jk}-(\beta+\gamma)_j g_{ik}-(\beta+\gamma)_k g_{ij}$$
which is easily shown to vanish if and only if $\alpha=0$ and $\beta=-\gamma$, and to be proportional to the metric $g_{jk}$ exactly when $\beta=-\gamma$; this latter situation is equivalent to $\nabla$ preserving the conformal class of $g$. 

It's additionally immediate from the non-metricity tensor that the canonical associated metric tensor is
$$\nabla^g=\nabla_g-B$$
where 
$$B_i{}^k{}_j=-\tfrac{(\beta-\gamma)_j}{2}\delta^k_i+\tfrac{(\beta-\gamma)^k}{2}g_{ij}$$
or (with indices lowered)
$$B_{ikj}=-\tfrac{(\beta-\gamma)_j}{2}g_{ik}+\tfrac{(\beta-\gamma)_k}{2}g_{ij}$$
This tensor has one independent trace,
$$B_j=B_{iij}=-\tfrac{n-1}{2}(\beta-\gamma)_j$$

\subsubsection{Curvature of $\nabla$}

The curvature of $\nabla$ is
\begin{align*}
R_{ij}{}^k{}_l=\,&\overline{R}_{ij}{}^k{}_l+(\partial_i \alpha_j-\partial_j \alpha_i)\delta^k_l+\overline{\nabla}_i \beta_l \delta^k_j-\overline{\nabla}_j \beta_l \delta^k_i +\overline{\nabla}_i \gamma^k g_{jl}-\overline{\nabla}_j \gamma^k g_{il}\\
&+(\delta^k_i \beta_j-\delta^k_j \beta_i)\beta_l+(\gamma_i g_{jl}-\gamma_j g_{il})\gamma^k+g(\beta, \gamma)(\delta^k_i g_{jl}-\delta^k_j g_{il})
\end{align*}
with index-lowered variant $R_{ijkl}=g_{km}R_{ij}{}^m{}_l$ given by
\begin{align*}
R_{ijkl}=&\overline{R}_{ijkl}+(\partial_i \alpha_j-\partial_j \alpha_i)g_{kl}+\overline{\nabla}_i \beta_l g_{kj}-\overline{\nabla}_j \beta_l g_{ki} +\overline{\nabla}_i \gamma_k g_{jl}-\overline{\nabla}_j \gamma_k g_{il}\\
&+(g_{ik} \beta_j-g_{jk} \beta_i)\beta_l+(\gamma_i g_{jl}-\gamma_j g_{il})\gamma_k+g(\beta, \gamma)(g_{ik} g_{jl}-g_{jk} g_{il})
\end{align*}
Since $\nabla$ is not necessarily torsion-free or metric, the curvature tensor lacks the familiar symmetries of a Riemannian curvature tensor, and there are three independent traces of $R$, not the usual one: the Ricci tensor $r_{jl}=R_{ij}{}^i{}_l$, $\rho_{jk}=R_{ijki}$, and $\zeta_{ij}=R_{ij}{}^k{}_k$

We first examine the Ricci tensor:
\begin{align*}
r=&\,\,\overline{r}+\mathrm{d}\alpha-\tfrac{1}{2}\mathrm{d}((n-1)\beta+\gamma)-\tfrac{1}{2}\pounds_{((n-1)\beta+\gamma)^{\#}} \,g\\
&+(n-1)\beta\otimes \beta-\gamma\otimes \gamma+(-\mathrm{d}^*_g(\gamma)+g((n-1)\beta+\gamma, \gamma))g
\end{align*}
Clearly the symmetric part of $r$ is
$$r_{()}=\overline{r}-\tfrac{1}{2}\pounds_{((n-1)\beta+\gamma)^{\#}} \,g+(n-1)\beta\otimes \beta-\gamma\otimes \gamma+(-\mathrm{d}^*_g(\gamma)+g((n-1)\beta+\gamma, \gamma))g$$
and the antisymmetric part is
$$r_{[\,]}=\mathrm{d}\alpha-\tfrac{1}{2}\mathrm{d}((n-1)\beta+\gamma)$$
with trace $s=g^{ij}r_{ij}=g^{ij}r_{(ij)}$ satisfying
$$s=\overline{s}+(n-1)\mathrm{d}^*_g(\beta-\gamma)+(n-1)(|\beta|^2+n\,g(\beta, \gamma)+|\gamma|^2)$$
Therefore the following equation holds for the traceless symmetric Ricci tensor $\mathring{r}_{()}=r_{()}-\tfrac{s}{n}g$:
\begin{align*}
\mathring{r}_{()}=&\,\,\overline{r}-\tfrac{1}{2}\pounds_{((n-1)\beta+\gamma)^{\#}} \,g+(n-1)\beta\otimes \beta-\gamma\otimes \gamma+\\
&-\tfrac{1}{n}(\overline{s}+\mathrm{d}^*_g((n-1)\beta+\gamma)+(n-1)|\beta|^2-|\gamma|^2)g
\end{align*}

Now consider the tensor $\rho$:
\begin{align*}
\rho=&-\overline{r}-\mathrm{d}\alpha-\tfrac{1}{2}\mathrm{d}(\beta+(n-1)\gamma)-\tfrac{1}{2}\pounds_{(\beta+(n-1)\gamma)^{\#}} \,g\\
&+\beta\otimes\beta-(n-1)\gamma\otimes\gamma+(-\mathrm{d}^*_g\beta-g(\beta, \beta+(n-1)\gamma))g
\end{align*}
The symmetric part is
$$\rho_{()}=-\overline{r}-\tfrac{1}{2}\pounds_{(\beta+(n-1)\gamma)^{\#}} \,g+\beta\otimes\beta-(n-1)\gamma\otimes\gamma+(-\mathrm{d}^*_g\beta-g(\beta, \beta+(n-1)\gamma))g$$
while the antisymmetric part is
$$\rho_{[\,]}=-\mathrm{d}\alpha-\tfrac{1}{2}\mathrm{d}(\beta+(n-1)\gamma)$$
and the trace $\sigma=g^{ij}\rho_{ij}=g^{ij}\rho_{(ij)}$ is fully determined by the trace of the Ricci tensor:
$$\sigma=-s$$
Thus for the traceless symmetric part we have
\begin{align*}
\mathring{\rho}_{()}=&\,-\overline{r}-\tfrac{1}{2}\pounds_{(\beta+(n-1)\gamma)^{\#}} \,g+\beta\otimes\beta-(n-1)\gamma\otimes\gamma\\
&-\tfrac{1}{n}(-\overline{s}+\mathrm{d}^*_g(\beta+(n-1)\gamma)+|\beta|^2-(n-1)|\gamma|^2)g
\end{align*}

Some interesting simplifications occur if we consider $\tfrac{1}{2}(r-\rho)$ and $\tfrac{1}{2}(r+\rho)$ (these are natural as traces of $R$ after breaking the $kl$ index pair into antisymmetric and symmetric parts). The traces are $s$ and $0$ respectively, while the other components are as follows:
\begin{align*}
\tfrac{1}{2}(\mathring{r}-\mathring{\rho})_{()}=&\,\,\overline{r}+\tfrac{n-2}{2}(-\tfrac{1}{2}\pounds_{(\beta-\gamma)^{\#}}g+\beta\otimes\beta+\gamma\otimes\gamma)\\
&-\tfrac{1}{n}(\overline{s}+\tfrac{n-2}{2}(\mathrm{d}^*_g(\beta-\gamma)+|\beta|^2+|\gamma|^2))g\\
\tfrac{1}{2}(r-\rho)_{[\,]}=&\,\,\mathrm{d}\alpha-\tfrac{n-2}{4}\mathrm{d}(\beta-\gamma)\\
\tfrac{1}{2}(r+\rho)_{()}=&\,\tfrac{1}{2}(\mathring{r}+\mathring{\rho})_{()}\\
=&-\tfrac{n}{4}\pounds_{(\beta+\gamma)^{\#}}g+\tfrac{n}{2}(\beta+\gamma)\cdot (\beta-\gamma)-\tfrac{1}{2}(\mathrm{d}^*_g(\beta+\gamma)+g(\beta+\gamma, \beta-\gamma))g\\
\tfrac{1}{2}(r+\rho)_{[\,]}=&-\tfrac{n}{4}\mathrm{d}(\beta+\gamma)
\end{align*}

Finally we have the antisymmetric tensor $\zeta$:
$$\zeta=n\mathrm{d}\alpha+2\mathrm{d}(\beta+\gamma)$$

\subsubsection{Curvature of $\nabla^g$}
The canonical metric connection of $\nabla$ is $\nabla^g$, whose Christoffel symbols $\eta_i{}^k{}_j$ satisfy (see Section 7.1)
$$\eta_i{}^k{}_j=\overline{\Gamma}_i{}^k{}_j+\tfrac{(\beta-\gamma)_j}{2}\delta^k_i-\tfrac{(\beta-\gamma)^k}{2}g_{ij}$$
This is the same type of connection that we considered above, with $\alpha=0$ and $\beta$ and $-\gamma$ replaced by $\tfrac{1}{2}(\beta-\gamma)$.
Denote the curvature of $\eta$ by $H=H_{ijkl}$. This tensor is antisymmetric in the $ij$ and $kl$ pairs since $\nabla^g$ is compatible with $g$, but we can't expect further symmetries. This is, however, enough to imply that there is only one independent trace; we'll use the Ricci trace over the $ik$ indices and call it $h$. We then have the formulas that follow.

\hspace{.1in}

\noindent (1,3)-Curvature Tensor:
\begin{align*}
H_{ij}{}^k{}_l=\,&\overline{R}_{ij}{}^k{}_l+\tfrac{1}{2}\overline{\nabla}_i (\beta-\gamma)_l \delta^k_j-\tfrac{1}{2}\overline{\nabla}_j (\beta-\gamma)_l \delta^k_i -\tfrac{1}{2}\overline{\nabla}_i (\beta-\gamma)^k g_{jl}+\tfrac{1}{2}\overline{\nabla}_j (\beta-\gamma)^k g_{il}\\
&+\tfrac{1}{4}(\delta^k_i (\beta-\gamma)_j-\delta^k_j (\beta-\gamma)_i)(\beta-\gamma)_l+\tfrac{1}{4}((\beta-\gamma)_i g_{jl}-(\beta-\gamma)_j g_{il})(\beta-\gamma)^k\\
&-\tfrac{1}{4}|\beta-\gamma|^2(\delta^k_i g_{jl}-\delta^k_j g_{il})
\end{align*}
(0,4)-Curvature Tensor:
\begin{align*}
H_{ijkl}=&\overline{R}_{ijkl}+\tfrac{1}{2}\overline{\nabla}_i (\beta-\gamma)_l g_{kj}-\tfrac{1}{2}\overline{\nabla}_j (\beta-\gamma)_l g_{ki} -\tfrac{1}{2}\overline{\nabla}_i (\beta-\gamma)_k g_{jl}+\tfrac{1}{2}\overline{\nabla}_j (\beta-\gamma)_k g_{il}\\
&+\tfrac{1}{4}g\owedge[(\beta-\gamma)\otimes(\beta-\gamma)]_{ijkl}-\tfrac{1}{2}|\beta-\gamma|^2 g\owedge g_{ijkl}
\end{align*}
Ricci tensor:
\begin{align*}
h=\,\,&\overline{r}-\tfrac{n-2}{4}\mathrm{d}(\beta-\gamma)-\tfrac{n-2}{4}\pounds_{(\beta-\gamma)^{\#}}g\\
&+\tfrac{n-2}{4}(\beta-\gamma)\otimes (\beta-\gamma)+(\tfrac{1}{2}\mathrm{d}^*_g(\beta-\gamma)-\tfrac{n-2}{4}|\beta-\gamma|^2)g
\end{align*}
Symmetric Ricci tensor:
\begin{align*}
h_{()}=\,\,&\overline{r}-\tfrac{n-2}{4}\pounds_{(\beta-\gamma)^{\#}}g+\tfrac{n-2}{4}(\beta-\gamma)\otimes (\beta-\gamma)\\
&+(\tfrac{1}{2}\mathrm{d}^*_g(\beta-\gamma)-\tfrac{n-2}{4}|\beta-\gamma|^2)g
\end{align*}
Antisymmetric Ricci tensor:
$$h_{[\,]}=-\tfrac{n-2}{4}\mathrm{d}(\beta-\gamma)$$
Trace:
$$\tau=\overline{s}+(n-1)\mathrm{d}^*_g(\beta-\gamma)-\tfrac{(n-1)(n-2)}{4}|\beta-\gamma|^2$$
Traceless symmetric Ricci tensor:
\begin{align*}
\mathring{h}_{()}=\,\,&\overline{r}-\tfrac{n-2}{4}\pounds_{(\beta-\gamma)^{\#}}g+\tfrac{n-2}{4}(\beta-\gamma)\otimes (\beta-\gamma)\\
&-\tfrac{1}{n}(\overline{s}+\tfrac{n-2}{2}\mathrm{d}^*_g(\beta-\gamma)+\tfrac{n-2}{4}|\beta-\gamma|^2)g
\end{align*}

\subsubsection{Specializing the equations}
If we substitute $\alpha=a\varphi$, $\beta=b\varphi$, $\gamma=c\varphi$ we obtain further specializations:

\hspace{.1in}

\noindent{\bf $\nabla$ tensors:}

\hspace{.1in}

\noindent (1,3)-Curvature Tensor:
\begin{align*}
R_{ij}{}^k{}_l=\,&\overline{R}_{ij}{}^k{}_l+a(\partial_i \varphi_j-\partial_j \varphi_i)\delta^k_l+b\overline{\nabla}_i \varphi_l \delta^k_j-b\overline{\nabla}_j \varphi_l \delta^k_i +c\overline{\nabla}_i \varphi^k g_{jl}-c\overline{\nabla}_j \varphi^k g_{il}\\
&+b^2(\delta^k_i \varphi_j-\delta^k_j \varphi_i)\varphi_l+c^2(\varphi_i g_{jl}-\varphi_j g_{il})\varphi^k+bc |\varphi|^2(\delta^k_i g_{jl}-\delta^k_j g_{il})
\end{align*}
(0,4)-Curvature Tensor:
\begin{align*}
R_{ijkl}=&\overline{R}_{ijkl}+a(\partial_i \varphi_j-\partial_j \varphi_i)g_{kl}+b\overline{\nabla}_i \varphi_l g_{kj}-b\overline{\nabla}_j \varphi_l g_{ki} +c\overline{\nabla}_i \varphi_k g_{jl}-c\overline{\nabla}_j \varphi_k g_{il}\\
&+b^2(g_{ik} \varphi_j-g_{jk} \varphi_i)\varphi_l+c^2(\varphi_i g_{jl}-\varphi_j g_{il})\varphi_k+bc |\varphi|^2(g_{ik} g_{jl}-g_{jk} g_{il})
\end{align*}
Ricci tensor:
\begin{align*}
r=&\,\,\overline{r}+a\mathrm{d}\varphi-\tfrac{1}{2}((n-1)b+c)\mathrm{d}\varphi-((n-1)b+c)\tfrac{1}{2}\pounds_{\varphi^{\#}} \,g\\
&+((n-1)b^2-c^2)\varphi\otimes \varphi+c(-\mathrm{d}^*_g\varphi+((n-1)b+c)|\varphi|^2)g
\end{align*}
Symmetric Ricci tensor:
$$r_{()}=\overline{r}-((n-1)b+c)\tfrac{1}{2}\pounds_{\varphi^{\#}} \,g+((n-1)b^2-c^2)\varphi\otimes \varphi+c(-\mathrm{d}^*_g\varphi+((n-1)b+c)|\varphi|^2)g$$
Antisymmetric Ricci tensor:
$$r_{[\,]}=a\mathrm{d}\varphi-\tfrac{1}{2}((n-1)b+c)\mathrm{d}\varphi$$
Trace of Ricci tensor:
\begin{align*}
s&=(\overline{s}+((n-1)b+c)\mathrm{d}^*_g\varphi+((n-1)b^2-c^2)|\varphi|^2)+n\cdot c(-\mathrm{d}^*_g\varphi+((n-1)b+c)|\varphi|^2)\\
&=\overline{s}+(n-1)(b-c)\mathrm{d}^*_g\varphi+(n-1)(b^2+nbc+c^2)|\varphi|^2
\end{align*}
Traceless symmetric Ricci tensor:
\begin{align*}
\mathring{r}_{()}=&\,\overline{r}-\tfrac{1}{2}((n-1)b+c)\pounds_{\varphi^{\#}} \,g+((n-1)b^2-c^2)\varphi\otimes \varphi+\\
&-\tfrac{1}{n}(\overline{s}+((n-1)b+c)\mathrm{d}^*_g\varphi+((n-1)b^2-c^2)|\varphi|^2)g
\end{align*}
$\rho$ tensor:
\begin{align*}
\rho=&-\overline{r}-a\mathrm{d}\varphi-\tfrac{1}{2}(b+(n-1)c)\mathrm{d}\varphi-\tfrac{1}{2}(b+(n-1)c)\pounds_{\varphi^{\#}} \,g\\
&+(b^2-(n-1)c^2)\varphi\otimes\varphi+b(-\mathrm{d}^*_g\varphi-(b+(n-1)c)|\varphi|^2)g
\end{align*}
Symmetric $\rho$:
$$\rho_{()}=-\overline{r}-\tfrac{1}{2}(b+(n-1)c)\pounds_{\varphi^{\#}} \,g+(b^2-(n-1)c^2)\varphi\otimes\varphi+b(-\mathrm{d}^*_g\varphi-(b+(n-1)c)|\varphi|^2)g$$
Antisymmetric $\rho$:
$$\rho_{[\,]}=-a\mathrm{d}\varphi-\tfrac{1}{2}(b+(n-1)c)\mathrm{d}\varphi$$
Trace of $\rho$:
$$\sigma=-s$$
Traceless symmetric $\rho$:
\begin{align*}
\mathring{\rho}_{()}=&\,-\overline{r}-\tfrac{1}{2}(b+(n-1)c)\pounds_{\varphi^{\#}} \,g+(b^2-(n-1)c^2)\varphi\otimes\varphi\\
&-\tfrac{1}{n}(-\overline{s}+(b+(n-1)c)\mathrm{d}^*_g\varphi+(b^2-(n-1)c^2)|\varphi|^2)g
\end{align*}
Linear combinations of $r$ and $\rho$:
\begin{align*}
\tfrac{1}{2}(\mathring{r}-\mathring{\rho})_{()}=&\,\,\overline{r}+\tfrac{n-2}{2}(-(b-c)\tfrac{1}{2}\pounds_{\varphi^{\#}}g+(b^2+c^2)\varphi\otimes\varphi)\\
&-\tfrac{1}{n}(\overline{s}+\tfrac{n-2}{2}((b-c)\mathrm{d}^*_g\varphi+(b^2+c^2)|\varphi|^2)g\\
\tfrac{1}{2}(r-\rho)_{[\,]}=&\,\,a\mathrm{d}\varphi-\tfrac{n-2}{4}(b-c)\mathrm{d}\varphi\\
\tfrac{1}{2}(r+\rho)_{()}=&\,\tfrac{1}{2}(\mathring{r}+\mathring{\rho})_{()}\\
=&(b+c)(-\tfrac{n}{4}\pounds_{\varphi^{\#}}g+\tfrac{n}{2}(b-c)\varphi\otimes \varphi-\tfrac{1}{2}(\mathrm{d}^*_g\varphi+(b-c)|\varphi|^2)g)\\
\tfrac{1}{2}(r+\rho)_{[\,]}=&-\tfrac{n}{4}(b+c)\mathrm{d}\varphi
\end{align*}
Antisymmetric $\zeta$ tensor:
$$\zeta=n\cdot a\, \mathrm{d}\varphi+2(b+c)\mathrm{d}\varphi$$

\vspace{.1in}

\noindent{\bf $\nabla^g$ tensors:}

\vspace{.1in}

\noindent (1,3)-Curvature Tensor:
\begin{align*}
H_{ij}{}^k{}_l=\,&\overline{R}_{ij}{}^k{}_l+\tfrac{b-c}{2}(\overline{\nabla}_i \varphi_l \delta^k_j-\overline{\nabla}_j \varphi_l \delta^k_i -\overline{\nabla}_i \varphi^k g_{jl}+\overline{\nabla}_j \varphi^k g_{il})\\
&+\tfrac{(b-c)^2}{4}(\delta^k_i \varphi_j-\delta^k_j \varphi_i)\varphi_l+\tfrac{(b-c)^2}{4}(\varphi_i g_{jl}-\varphi_j g_{il})\varphi^k\\
&-\tfrac{(b-c)^2}{4}|\varphi|^2(\delta^k_i g_{jl}-\delta^k_j g_{il})
\end{align*}
(0,4)-Curvature Tensor:
\begin{align*}
H_{ijkl}=&\overline{R}_{ijkl}+\tfrac{b-c}{2}(\overline{\nabla}_i \varphi_l g_{kj}-\overline{\nabla}_j \varphi_l g_{ki} -\overline{\nabla}_i \varphi_k g_{jl}+\overline{\nabla}_j \varphi_k g_{il})\\
&+\tfrac{(b-c)^2}{4}g\owedge[\varphi\otimes\varphi]_{ijkl}-\tfrac{(b-c)^2}{2}|\varphi|^2 g\owedge g_{ijkl}
\end{align*}
Ricci tensor:
\begin{align*}
h=\,\,&\overline{r}-\tfrac{n-2}{4}(b-c)\mathrm{d}\varphi-\tfrac{n-2}{4}(b-c)\pounds_{\varphi^{\#}}g\\
&+\tfrac{n-2}{4}(b-c)^2\varphi\otimes \varphi+(\tfrac{b-c}{2}\mathrm{d}^*_g\varphi-\tfrac{n-2}{4}(b-c)^2|\varphi|^2)g
\end{align*}
Symmetric Ricci tensor:
\begin{align*}
h_{()}=\,\,&\overline{r}-\tfrac{n-2}{4}(b-c)\pounds_{\varphi^{\#}}g+\tfrac{n-2}{4}(b-c)^2\varphi\otimes \varphi\\
&+(\tfrac{b-c}{2}\mathrm{d}^*_g\varphi-\tfrac{n-2}{4}(b-c)^2|\varphi|^2)g
\end{align*}
Antisymmetric Ricci tensor:
$$h_{[\,]}=-\tfrac{n-2}{4}(b-c)\mathrm{d}\varphi$$
Trace:
$$\tau=\overline{s}+(n-1)(b-c)\mathrm{d}^*_g\varphi-\tfrac{(n-1)(n-2)}{4}(b-c)^2|\varphi|^2$$
Traceless symmetric Ricci tensor:
\begin{align*}
\mathring{h}_{()}=\,\,&\overline{r}-\tfrac{n-2}{4}(b-c)\pounds_{\varphi^{\#}}g+\tfrac{n-2}{4}(b-c)^2\varphi\otimes \varphi\\
&-\tfrac{1}{n}(\overline{s}+\tfrac{n-2}{2}(b-c)\mathrm{d}^*_g\varphi+\tfrac{n-2}{4}(b-c)^2|\varphi|^2)g
\end{align*}

\subsubsection{Assuming $a=b$, $(n-1)b+c=-1$, $(n-1)b^2-c^2=-\tfrac{1}{m}$} If we assume $a=b$, $(n-1)b+c=-1$, and $(n-1)b^2-c^2=-\tfrac{1}{m}$, then when $n>2$ we have two solutions $(b_+, c_+)$ and $(b_-, c_-)$, given explicitly by
$$b_{\pm}=\tfrac{1}{n-2}\big(-1\pm\sqrt{\tfrac{m+n-2}{m(n-1)}}\,\big)$$
$$c_{\pm}=-\tfrac{1}{n-2}\big(-1\pm\sqrt{\tfrac{(m+n-2)(n-1)}{m}}\,\big)$$
This produces two torsion-free connections
$$\nabla_+=\nabla_g+b_+\varphi_i \delta^k_j+b_+\varphi_j \delta^k_i+c_+\varphi^k g_{ij}$$
$$\nabla_-=\nabla_g+b_-\varphi_i \delta^k_j+b_-\varphi_j \delta^k_i+c_-\varphi^k g_{ij}$$
with identical and particularly simple traceless symmetric Ricci tensors
$$\mathring{r}_{\pm}{}_{()}=\overline{r}+\tfrac{1}{2}\pounds_{\varphi^{\#}} \,g-\tfrac{1}{m}\varphi\otimes \varphi-\tfrac{1}{n}(\overline{s}-\mathrm{d}^*_g\varphi-\tfrac{1}{m}|\varphi|^2)g$$
despite the fact that their symmetric Ricci tensors differ due to an explicit dependence on $c$:
$$r_{\pm}{}_{()}=\overline{r}+\tfrac{1}{2}\pounds_{\varphi^{\#}} \,g-\tfrac{1}{m}\varphi\otimes \varphi+c_{\pm}(-\mathrm{d}^*_g\varphi-|\varphi|^2)g$$
However, if one considers the average of the symmetric Ricci tensors,
$$\tfrac{1}{2}(r_+{}_{()} +r_-{}_{()})=\overline{r}+\tfrac{1}{2}\pounds_{\varphi^{\#}} \,g-\tfrac{1}{m}\varphi\otimes\varphi+\tfrac{1}{n-2}(-\mathrm{d}^*_g\varphi-|\varphi|^2)g$$
(using $\tfrac{1}{2}(c_+ +c_-)=\tfrac{1}{n-2}$) then this does not depend on choosing between $(b_+, c_+)$ and $(b_-, c_-)$. From this we find also
$$\tfrac{1}{2}(s_+ +s_-)=\overline{s}-2\mathrm{d}^*_g\varphi-\tfrac{m+1}{m}|\varphi|^2-\tfrac{2}{n-2}(\mathrm{d}^*_g\varphi+|\varphi|^2)$$
The expression $\overline{s}-2\mathrm{d}^*_g\varphi-\tfrac{m+1}{m}|\varphi|^2$ is an analogue of the notion of weighted scalar curvature in \cite{Case}.

Perhaps studying the averages of these and other curvature tensors will prove worthwhile; it's not hard to see that the quantities $b_+ +b_-$, $c_+ +c_-$, $b_+^2+b_-^2$, $c_+^2+c_-^2$, and $b_+ c_+ +b_- c_-$ don't involve any root extraction, so the averages of curvature tensors won't involve any root expressions. However, it's not entirely clear what underlying principle would lead to the consideration of such averages. We leave such explorations for the interested reader to embark upon.

\section{Acknowledgements} The author would like to thank the Rutgers University, New Brunswick Mathematics Department for their financial support while this work was completed.

\end{document}